\documentclass{amsart} 
\usepackage{amssymb}
\usepackage[all]{xy}
\usepackage[bbgreekl]{mathbbol}
\usepackage{a4wide}	

\DeclareSymbolFontAlphabet{\mathbb}{AMSb}	
\DeclareSymbolFontAlphabet{\mathbbl}{bbold}	%

\newtheorem{thm}{Theorem}[section]
\newtheorem{lem}[thm]{Lemma}
\newtheorem{prop}[thm]{Proposition}
\newtheorem{cor}[thm]{Corollary}

\theoremstyle{definition}
\newtheorem{df}[thm]{Definition}
\theoremstyle{remark}
\newtheorem{rem}[thm]{Remark}

\newcommand{\id}{\operatorname{id}}
\newcommand{\cO}{\mathcal{O}}

\newcommand{\dS}{\hat{S}}
\newcommand{\de}{\hat{\varepsilon}}
\newcommand{\cK}{\mathcal{K}}
\newcommand{\KK}{KK}

\newcommand{\End}{\operatorname{End}}
\newcommand{\dG}{\QG}

\newcommand{\dlambda}{\hat{\lambda}}
\newcommand{\atimes}{\otimes_{\rm alg}}

\newcommand{\dD}{\hat{\Delta}}
\newcommand{\dtau}{\hat{\tau}}
\newcommand{\dphi}{\hat{\psi}}
\newcommand{\cD}{\mathcal{D}}
\newcommand{\cE}{\mathcal{E}}
\newcommand{\cF}{\mathcal{F}}
\newcommand{\cP}{\mathcal{P}}
\newcommand{\cN}{\mathcal{N}}
\newcommand{\op}{{\rm op}}

\newcommand{\G}{\mathbb{G}}
\newcommand{\HG}{\mathbb{H}}

\newcommand{\Pol} {\cO}

\newcommand{\bn}{\mathbb N}

\newcommand{\Linf}{\ell^{\infty}}

\newcommand{\GGamma}{\mathbbl{\Gamma}}

\newcommand {\Irr} {\textup{Irr}\,}

\newcommand{\wot}{\overline{\ot}}

\newcommand{\Com}{\Delta}

\newcommand{\QG}{\GGamma}

\newcommand{\hQG}{\widehat{\GGamma}}

\newenvironment{rlist}
{
	
	\begin{enumerate}}
	{\end{enumerate}}

\newcommand{\cst}{\ifmmode\mathrm{C}^*\else{$\mathrm{C}^*$}\fi}

\newcommand{\ot}{\otimes}

\newcommand{\ltwo}{\ell^2(\QG)}
\newcommand{\lone}{\ell^1(\QG)}

\numberwithin{equation}{section}
\DeclareMathOperator{\C}{C}

\begin{document}
\title[Quantum Baum-Connes and Rosenberg Conjectures]{On the Baum--Connes conjecture for discrete quantum groups with torsion and the quantum Rosenberg Conjecture}
\author{Yuki Arano}
\address{Graduate School of Science, Kyoto University, Sakyo-ku, Kyoto, 606-8216, Japan}
\email{y.arano@math.kyoto-u.ac.jp}
\author{Adam Skalski}
\address{Institute of Mathematics, Polish Academy of Sciences, ul.\'Sniadeckich 8, 00-656 Warsaw, Poland}
\email{a.skalski@impan.pl}


\begin{abstract}
	We give a decomposition of the equivariant Kasparov category for discrete quantum group with torsions. As an outcome, we show that the crossed product by a discrete quantum group in a certain class preserves the UCT. We then  show that quasidiagonality of a reduced  \cst-algebra of a countable discrete quantum group $\GGamma$ implies that $\GGamma$ is amenable, and deduce from the work of Tikuisis, White and Winter, and the results in the first part of the paper, the converse (i.e.\ the quantum Rosenberg Conjecture) for a large class of countable discrete unimodular quantum groups. We also note that the unimodularity is a necessary condition.
\end{abstract}
\maketitle
	\section{Introduction}
	In \cite{UCT}, Rosenberg and Schochet have introduced a property of C*-algebras called the Universal Coefficient Theorem (UCT in short) for $K$-theory of C*-algebras and  have shown that it holds  for all C*-algebras in the so-called bootstrap class. The UCT gives a formula computing the $KK$-groups only from the $K$-groups. This property plays an important role in the classification of nuclear C*-algebras (see e.g.\ \cite{TWWUCT} and the last section of this paper).

	The UCT for group C*-algebras is related to (a variation of) the Baum--Connes conjecture of groups. In \cite{Tu}, Tu proved that the group C*-algebra of a discrete group with Haagerup property satisfies the UCT using the Higson--Kasparov type argument \cite{EKK} for groupoids.

	The Baum--Connes conjecture for quantum groups first appeared in the series of works of Meyer and Nest \cite{MeyerNest3}, \cite{Meyer} (after an early paper \cite{GoswamiKuku}). Even though there is no unified method proving the Baum--Connes conjecture for fairly general quantum groups, it has been established for many known examples of discrete quantum groups: \cite{FreslonMartos}, \cite{VergniouxVoigt}, \cite{VoigtBC}, \cite{VoigtQAut}.

In this paper, we study the general theory of the Baum--Connes conjecture for discrete quantum group with possible torsion. In particular we give a decomposition of the equivariant category $KK^{\G}$ (where $\G$ is any compact quantum group), which gives the Baum--Connes assembly map. As a byproduct of the general theory, we prove that the group C*-algebra of a discrete quantum group satisfying the Baum--Connes conjecture satisfies also the UCT. This is applied in the last section of the paper to the considerations regarding the quantum version of the Rosenberg Conjecture, connecting amenability of a discrete group to quasidiagonality of its C*-algebra.

The detailed plan of the paper is as follows: in the following section we introduce the notation and some background related to discrete/compact quantum groups and triangulated categories. In Section 3 we present a `crossed product type' construction for two  C*-algebras equipped respectively with left and right action of a given compact quantum group, which is then applied in Section 4 to build an adjoint functor between certain $KK$-categories. In Section 5 we establish as a consequence a relationship between the $\langle {\rm Cof} \rangle$-Baum--Connes property of a discrete quantum group and the Universal Coefficient Theorem for some crossed products. Finally in Section 6 the applications to quantum Rosenberg Conjecture are discussed.
	
	\subsection*{Acknowledgment}
This work was initiated in the workshop ``The 6th Workshop on
Operator Algebras and their Applications" in the School of Mathematics of Institute for Research in Fundamental Sciences (IPM). The authors would like to thank the organizers and IPM for their hospitality.
Y.A.\ is supported by JSPS KAKENHI Grant Number JP18K13424.
A.S.\ was partially supported by the National Science Centre (NCN) grant no.~2014/14/E/ST1/00525. He acknowledges discussions with Pawe\l$\,$ J\'oziak, Piotr So\l tan, Stuart White and Joachim Zacharias on the subject of the last section of the paper.

Finally we thank the referee for a careful reading of our paper and several useful comments.

	\section{Preliminaries}
	\subsection{Quantum groups}
	
	Let $\QG$ be a discrete quantum group (so that $\hQG$ is a compact quantum group in the sense of \cite{wor} -- we refer to that paper for the details of the facts introduced below, and often write simply $\G$ for the dual compact quantum group). We study $\QG$ via its algebra of functions, $c_0(\QG)$. Recall that
	\[ c_0(\QG) = \bigoplus_{\alpha\in\Irr_{\hQG}} M_{n_{\alpha}},\]
	where $\Irr_{\hQG}$ denotes the set of equivalence classes of irreducible unitary representations of $\hQG$, which will be usually assumed to be countable in this paper (in which case we will say that $\hQG$ has countable dual). The span of coefficients of $\Irr_{\hQG}$ is a Hopf *-algebra denoted $\Pol(\hQG)$, admitting a Haar (bi-invariant) state $h$. Note that we will also write $c_c(\QG)$ for the algebraic direct sum:
		\[ c_c(\QG) = \bigoplus^{\textup{alg}}_{\alpha\in\Irr_{\hQG}} M_{n_{\alpha}},\]

	The \cst-algebra $\cst_r(\GGamma)$, often written as $\C(\hQG)$, is the \cst-completion of $\Pol(\hQG)$ in the GNS representation with respect to $h$. For each $\alpha \in \Irr_{\hQG}$ we choose a representative, i.e.\ a unitary matrix $U^{\alpha} = (u^{\alpha}_{i,j})_{i,j=1}^{n_{\alpha}} \in M_{n_{\alpha}}(\cst_r(\GGamma))$. We may and do assume that $c_0(\QG)$ is represented, via the left regular representation of $\QG$,  on the Hilbert space $\ell^2(\QG)$ (viewed here as the GNS space of the Haar state of $\hQG$, so also the space on which $\cst_r(\GGamma)$ acts); this representation will be later denoted by $\dlambda$. The matrix units in $M_{n_{\alpha}}\subset c_0(\QG)$ will be denoted by $e_{i,j}^{\alpha}$.
	
	The multiplicative unitary of $\QG$ is the unitary $W\in B(\ltwo \ot \ltwo)$ given by the formula:
	\[ W = \sum_{\alpha \in \Irr_{\hQG}} \sum_{i,j=1}^{n_{\alpha}}e_{j,i}^{\alpha} \ot (u^{\alpha}_{i,j})^*\]
	The von Neumann completion of $c_0(\QG)$ 
	will be denoted by $\Linf(\QG)$. The predual of $\Linf(\QG)$
	will be denoted by $\lone$.
	
	The coproduct of $\QG$, a coassociative  normal unital $^*$-homomorphism $\Com:\Linf(\QG) \to \Linf(\QG) \wot \Linf(\QG)$ is implemented by $W$ via the following formula:
	\begin{equation}
	\Com(x) = W^* (1 \ot x) W, \;\;\; x \in \Linf(\QG).
	\end{equation}
	Given a functional $\phi \in \lone$ we define the (normal, bounded) maps $L_\phi:\Linf(\QG) \to \Linf(\QG)$ and $R_\phi:\Linf(\QG) \to \Linf(\QG)$  via the formulas
	\[ L_\phi= (\phi \ot \id)\circ \Com, \;\;\; R_\phi= (\id \ot \phi)\circ \Com.\]

A discrete quantum group $\QG$ is said to be \emph{finite}, if $\Irr_{\hQG}$ is finite (equivalently, $c_0(\QG)$ is finite-dimensional), and \emph{countable}, if $\Irr_{\hQG}$ is countable (equivalently, $c_0(\QG)$ is separable). Further $\QG$ is \emph{unimodular} if its left and right Haar weights coincide; equivalently the Haar state $h$ of $\hQG$ is tracial.


A discrete quantum group $\QG$ is called \emph{amenable} if it admits a bi-invariant mean, i.e.\ a state $m\in \Linf(\QG)^*$, such that for all $\phi \in \lone$ there is
	\[m \circ L_{\phi} = m \circ R_{\phi} = \phi(1) m.\]
	
By \cite{Vaesetal} a discrete quantum group $\QG$ is amenable if it admits a left invariant mean $m\in \Linf(\QG)^*$: a state such that for each $\phi \in \lone$ there is $m \circ L_{\phi} =  \phi(1) m$. In fact it suffices to check the last formula for the functionals of the form $\widehat{e_{i,j}^{\alpha}}$, $\alpha \in \Irr_{\hQG}$, $i,j=1,\ldots,n_{\alpha}$, as the latter are linearly dense in $\lone$, and the map $\phi \mapsto L_\phi$ is a (complete) isometry. Thus we will need the following explicit form of the map $L_{\phi}$ for $\phi=\widehat{e_{i,j}^{\alpha}}$:
\begin{equation} \label{Lphi}
	L_{\phi} (x) =
	\sum_{p=1}^{n_{\alpha}} u^{\alpha}_{i,p}  x (u^{\alpha}_{j,p})^*
\end{equation}
(with $x \in \ell^{\infty}(\QG)$).

Recall that $\G$ denotes the dual compact quantum group of $\QG$.
The left regular representation $\lambda \colon C(\G) \to B(L^2(\G))$ is the GNS representation with respect to the Haar state $\varphi$; note that $L^2(\G)$ is canonically isomorphic to $\ell^2(\GGamma)$. We also have the right regular representation $\rho(x) = J R(x)^* J$, $x \in C(\G)$, where $J$ is the modular conjugation and $R$ is the unitary antipode.

Via the natural pairing
	\[\cO(\G) \times c_c(\dG) \to \mathbb{C},\]
we put a (multiplier) Hopf algebra structure on $c_c(\dG)$.

For details of quantum group actions and the associated crossed products we refer for example to \cite{Kenny} and \cite{Vaes}; note that we always work with reduced/faithful actions. Given a left action $\alpha \colon A \to C(\G) \otimes A$ we call $A$ a \emph{$\G$}-C*-\emph{algebra}. Such an action induces a right $c_c(\dG)$-comodule algebra structure on $A$:
\[a \triangleleft x := (x \otimes \id)\alpha(a)\]
for $a \in A$ and $x \in c_c(\dG)$.	Similarly a right action $\beta \colon B \to B \otimes C(\G)$ induces a left $c_c(\dG)$-comodule algebra structure on $B$:
	\[x \triangleright b := (\id \otimes x)\beta(b)\]
for $b \in B$ and $x \in c_c(\dG)$.

For a finite dimensional C*-algebra $D$ and a left action $\alpha$ of $\G$ on $D$, there always exists a $\G$-invariant faithful state $\varphi_D$ on $D$,  which is of the form $\varphi_D = {\rm Tr}(\rho \cdot)$, where ${\rm Tr}$ is the trace  taking value 1 at each minimal projection and $\rho \in D_+$. Let $(\lambda_D, L^2(D) = L^2(D,\varphi_D),\Omega_D)$ be the GNS representation of $\varphi_D$. Then $L^2(D)$ also carries a $*$-representation $\rho_D$ of the opposite C*-algebra $D^\op$, given by
	\[\rho_D(x^\op) \lambda_D(y) \Omega_D = \lambda_D(y \rho^{1/2} x \rho^{-1/2}) \Omega_D,\;\;\; x, y \in D.\]
	Furthermore the formula
	\[U^*(a \otimes x \Omega) = \alpha(x) (a \otimes \Omega), \;\;\; a \in C(\G), x \in D,\]
	defines a unitary representation $U \in C(\G) \otimes B(L^2(D))$.

	\begin{df}\cite{BS}
	For a C*-algebra $A$ with a $\G$-action $\alpha \colon A \to C(\G) \otimes A$ and a Hilbert $A$-module $\cE$, a \emph{$\G$-action} on $\cE$ is a linear map $\alpha_\cE \colon \cE \to C(\G) \otimes \cE$ such that
	\begin{enumerate}
	\item $\alpha_\cE(x a) = \alpha_\cE(x) \alpha(a)$ for all $x \in \cE, a \in A$,
	\item the linear span $\alpha_\cE(\cE)(C(\G) \otimes 1)$ is dense in $C(\G) \otimes \cE$ and
	\item $(\id \otimes \alpha_\cE)\alpha_\cE = (\Delta \otimes \id)\alpha_\cE$.
	\end{enumerate}
	\end{df}
	This is equivalent to say that $\alpha_\cE$ is the corner of an action of $\G$ on the linking algebra $\cK(\cE \oplus A) \simeq \left(\begin{matrix} \cK(\cE) & \cE \\ \cE^* & A \end{matrix} \right)$ which coincides with $\alpha$ on $A \cong \left(\begin{matrix} 0 & 0 \\ 0 & A \end{matrix} \right)$. See \cite{BS} for details.

Finally for the notion of \emph{torsion} in the context of compact quantum groups we refer for example to \cite{YukiKenny}.

	\subsection{Triangulated category}
	In \cite{MeyerNest3}, Meyer and Nest introduced a framework to work on $KK$-theory in terms of triangulated categories. In this section, we review their work, not going into the full generality of triangulated categories but only restricting ourselves to describe the situation in terms of $\G$-equivariant $KK$-theory, where $\G$ is a fixed compact quantum group with a countable dual.

	To each equivariant $*$-homomorphism $\varphi \colon A \to B$, one can associate an exact sequence called the mapping cone exact sequence:
	\[0 \to C_\varphi \xrightarrow{\iota} M_\varphi \xrightarrow{{\rm ev}_1} B \to 0\]
	where $M_\varphi = \{(f,a) \in (C[0,1] \otimes B) \oplus A\} \mid f(1) = \varphi(a) \} \supset C_\varphi = \{(f,a) \in M_\varphi \mid f(0) = 0\}$.
	Notice that $M_\varphi$ is homotopy equivalent to $A$.
	We may continue this construction for $\iota$ to get another mapping cone exact sequence:
	\[0 \to C_\iota \to M_{\iota} \to M_\varphi \to 0.\]
	Then $C_\iota$ is actually homotopy equivalent to the suspension $SB = C_0(\mathbb{R}) \otimes B$. Hence in the category $KK^{\G}$, we get a diagram
	\[\dots \to SC_\varphi \to SA \to SB \to C_\varphi \to A \to B\]
	or
	\[\xymatrix{A \ar[rr]^\varphi && B \ar[ld] \\ & C_\varphi \ar[lu]|\circ &}\]
	which gives the six-term exact sequence after taking the $K$-groups. (Here the circle on the arrow represents the change of the degree.) A \emph{distinguished triangle} is a diagram which is $KK^{\G}$-equivalent to some mapping cone triangle as above.
%
	\begin{df}
	A \emph{localizing subcategory} of $KK^{\G}$ is a full subcategory which is closed under taking countable direct sums, suspensions and mapping cones.

	Let $\cP,\cN$ be localizing subcategories of $KK^{\G}$. We say that the pair $(\cP,\cN)$ is \emph{complementary} if
	\begin{enumerate}
	\item $KK^{\G}(P,N) = 0$ for any $P \in \cP$ and $N \in \cN$.
	\item For any $A \in KK^{\G}$, there exists a distinguished triangle
	\[\xymatrix{P(A) \ar[rr] && A \ar[ld] \\ & N(A) \ar[lu]|\circ &},\]
	where $P(A) \in \cP$ and $N(A) \in \cN$.
	\end{enumerate}
	\end{df}
	\begin{rem}
	For arbitrary choice of $P(A)$ and $N(A)$ as above, the morphism $P(A) \to A$ is universal among all morphisms $P \to A$ for $P \in \cP$. In fact, we write the six-term exact sequence of $KK^{\G}$:
	\[\dots \to KK^{\G}(P,SN(A)) \to KK^{\G}(P,P(A)) \to KK^{\G}(P,A) \to KK^{\G}(P,N(A)) \to \dots.\]
	Since $KK^{\G}(P,SN(A)) = KK^{\G}(P,N(A)) = 0$ by (1), the map $KK^{\G}(P,P(A)) \to KK^{\G}(P,A)$ is an isomorphism. This is what we claimed.

	This in particular shows that the triangle $P(A) \to A \to N(A)$ is unique up to isomorphism.
	\end{rem}

	The following result holds in a  more general setting, namely, when the adjoint is only partially defined, but we only use it in the following form.
	\begin{thm}\label{thm:MN}\cite[Theorem 3.31]{MeyerNest3}
	Let $\HG, \G$ be compact quantum groups with countable duals and let $F_i$ be a countable family of functors $KK^{\G} \to KK^{\HG}$ which preserve the distinguished triangles. Assume that there exist left adjoint functors $F_i^\perp \colon KK^{\HG}\to KK^{\G}$, i.e.\ $KK^{\G}(F_i^\perp(A),B) \simeq KK^{\HG}(A,F_i(B))$ for all C*-algebras $A, B$ equipped respectively with a $\G$ and $\HG$ action.
	We set $\cP$ to be the smallest thick subcategory containing $F_i^\perp(A)$ and $\cN$ to be the full subcategory whose object is $N \in KK^{\G}$ such that $F_i(N)$ is $KK^{\HG}$-contractible (note that $\cN$ is automatically thick). Then $(\cP,\cN)$ is localizing.
	\end{thm}
	In this case, an explicit construction of $P(A)$ is given by the phantom castle construction \cite[Section 3]{MeyerNest3}. We recall the construction in Section \ref{section UCT}.

	\section{Crossed products}
	Let $\G$ be a compact quantum group. For two C*-algebras with $\G$-actions, there is no general way of constructing the ``product" action on the tensor product. However it is possible to construct a C*-algebra ressembling a ``crossed product" with respect to the product action. The construction works for arbitrary locally compact quantum group actions in an analogous manner, but we restrict ourselves to work with the compact case.

Let $A$ (resp.\ $B$) be a C*-algebra with a right (resp.\ left) $\G$-action:
	\[\alpha \colon A \to A \otimes C(\G), \beta \colon B \to C(\G) \otimes B;\]
these will be fixed throughout this section. By	a \emph{covariant representation} of $(A,B,\G)$ on a Hilbert space $H$ we understand a triple of representations $\pi_A \colon A \to B(H)$, $\pi_B \colon B \to B(H)$ and $U \in M(C(\G) \otimes K(H))$, a unitary representation of $\G$, which satisfies the following conditions.
\begin{itemize}
	\item $\pi_A(A)$ and $\pi_B(B)$ commute;
	\item $U^*(1 \otimes \pi_B(b)) U = (\id \otimes \pi_B)\beta(b)$ for any $b \in B$;
	\item $\sigma(U)(\pi_A(a) \otimes 1)\sigma(U)^* = (\pi_A \otimes \id)\alpha(a)$ for any $a \in A$ (where $\sigma$ denotes the tensor flip).
\end{itemize}
	Recall that we denote the dual discrete quantum group of $\G$ by $\QG$. 
	Take the algebraic cores \cite[Definition 3.15]{Kenny} $A_0$ and $B_0$ of $A$ and $B$.
	We define a $*$-algebra $\mathcal{A} = A_0 \rtimes_{\rm alg} \G \ltimes_{\rm alg} B_0$ as follows.
	\begin{itemize}
	\item As a vector space,  $\mathcal{A}$ is isomorphic to $A_0 \atimes c_c(\dG) \atimes B_0$. The element in $\mathcal{A}$ corresponding to $a \otimes x \otimes b$ is denoted by $a x b$ for $a \in A_0, x \in c_c(\dG), b \in B_0$.
	\item The product is given by $(a x b)(a' x' b') = a (x_{(1)} \triangleright a') x_{(2)} x'_{(1)} (b \triangleleft x'_{(2)}) b'$ for $a,a' \in A_0, x,x' \in c_c(\dG)$, $b,b' \in B_0$. Notice that the sum on the right hand side is finite since $a'$ and $b$ are in the respective algebraic cores.
	\end{itemize}
	Let $A \rtimes \G \ltimes B$ be the universal C*-completion of $\mathcal{A}$.
	\begin{rem}
	In the von Neumann algebra setting, a similar construction arises from Popa's symmetric enveloping algebra \cite{Popa} (or the Longo--Rehren inclusion \cite{LongoRehren}) of a subfactor of the form $M^\G \subset M$ for a minimal action of $\G$ on a factor $M$.
	\end{rem}
	\begin{prop}
	We have the following.
	\begin{itemize}
	\item The C*-algebras $A,B$ and $c_0(\dG)$ are nondegenerate C*-subalgebras in the multiplier C*-algebra $M(A \rtimes \G \ltimes B)$.
	\item There is a natural one-to-one correspondence between the covariant representations of $(A,B,\G)$ and $*$-representations of $A \rtimes \G \ltimes B$.
	\end{itemize}
	\end{prop}
	\begin{proof}
	Let $A^m$ (resp.\ $B^m$) be the universal C*-envelope of $A_0$ (resp.\ $B_0$) equipped with a right (resp.\ left) universal action of $\G$, to be denoted $\alpha^u$. We first observe that $A^m$ coincides with the maximalization of $A$ in the sense of \cite[Definition 6.1]{Fischer} (See \cite{EQ} for the group case).
	To this end, we only need to show
	\[A \rtimes \G \rtimes \dG \simeq A^m \otimes K(L^2(\G))\]
	(see the proof of \cite[Theorem 6.4]{Fischer}).
	Since $(\id \otimes \varphi(\cdot)1)\circ  \alpha^u \colon A^m \otimes C(\G) \to (A^m)^\alpha \ot 1= A^\alpha \ot 1$ is a faithful conditional expectation, the action map $A_0 \to A_0 \otimes_{\rm alg} \cO(\G)$ extends to
	\[\alpha^m \colon A \to A^m \otimes C(\G).\]
	Now the triplet $((\id \otimes \rho)\alpha^m, \id_{A^m} \otimes \hat \lambda, \id_{A^m}
 \otimes \lambda)$ gives a $*$-homomorphism
	\[A \rtimes \G \rtimes \dG \to A^m \otimes K(L^2(\G)).\]
	Conversely the algebraic crossed product $A_0 \rtimes_{\rm alg} \G \rtimes_{\rm alg} \dG$ is isomorphic to $A_0 \otimes_{\rm alg} F(L^2(\G))$ where $F(L^2(\G)) = {\rm span} \{xy \in x \in \cO(\G), y \in c_c(\dG)\} \subset K(L^2(\G))$ is the $*$-algebra of all finite rank operators supported on finitely many components in $\Irr_\G$. Hence the universal completion of $A_0 \rtimes_{\rm alg} \G \rtimes_{\rm alg} \dG$ is naturally isomorphic to $A^m \otimes K(L^2(\G))$ and hence the $*$-homomorphism $A_0 \rtimes_{\rm alg} \G \rtimes_{\rm alg} \dG \to A \rtimes \G \rtimes \dG$ induces a map $A^m \otimes K(L^2(\G)) \to A \rtimes \G \rtimes \dG$. Since the two  maps described above are inverse to each other, we get the conclusion.
	In particular, the natural map $A^m \rtimes \G \to A \rtimes \G$ is an isomorphism.

	Since $A_0$ and $c_c(\dG)$ satisfy the commutation relation as in $A_0 \rtimes_{\rm alg} \G$, we obtain a nondegenerate $*$-homomorphism $A \rtimes \G \simeq A^m \rtimes \G \to M(A \rtimes \G \ltimes B)$. Since there exists a non-degenerate $*$-homomorphism
	\[A \rtimes \G \ltimes B \to M((A \rtimes \G) \otimes (\G \ltimes B)) \colon a x b \mapsto (a \otimes 1) \dD(x) (1 \otimes b),\]
	the map is injective. Similarly we get a nondegenerate injective $*$-homomorphism $B \to M(A \rtimes \G \ltimes B)$. This proves the first assertion in the proposition.

	For the second assertion, by definition of $\mathcal{A}$, we obtain a $*$-representation of $\mathcal{A}$ from a covariant representation of $(A,B,\G)$. Conversely the covariant representation of $(A,B,\G)$ is obtained by the first assertion from a $*$-representation of $\mathcal{A}$.
	\end{proof}
	\begin{lem}\label{lem ce}
	Let $\varphi_A$ be a $\G$-invariant state. Then $\varphi_A$ induces a conditional expectation
	\[A \rtimes \G \ltimes B \to \G \ltimes B \colon a x b \mapsto \varphi_A(a) x b,\]
	where $a \in A$, $x \in c_c(\dG)$, $b \in B$.
	\end{lem}
	\begin{proof}
	Take the GNS construction $(L^2(A,\varphi_A),\Omega)$ for $\varphi_A$.
	Consider the Hilbert $ \G \ltimes B$-module $\cE = L^2(A,\varphi_A) \otimes \G \ltimes B$. Then $A \rtimes \G \ltimes B$ admits a representation on $\cE$ defined by (the continuous extension of) the formula
	\[(axb)(a'\Omega \otimes x' b') = a (x_{(1)} \triangleright a') \Omega \otimes x_{(2)} x'_{(1)} (b \triangleleft x'_{(2)}) b'\]
	for $a, a' \in A_0$, $x,x' \in c_c(\dG)$, $b,b' \in B_0$, where again $A_0$ and $B_0$ denote the respective algebraic cores. Take an approximate unit $(e_i)_{i \in I}$ of $\G \ltimes B$. Then the desired conditional expectation is given by
	\[x \mapsto \lim_{i \in I} (\Omega \otimes e_i, x(\Omega \otimes e_i)),\]
	hence it is well-defined.
	\end{proof}
	Using this lemma, we give an easy structural result on this crossed product for later use.
	\begin{prop}\label{crossed fd}
	Let $D$ be a finite dimensional C*-algebra with a right $\G$-action and $B$ be a separable C*-algebra with a left $\G$-action.
	\begin{enumerate}
	\item If $B$ is finite dimensional, then the C*-algebra $D \rtimes \G \ltimes B$ is a direct sum of matrix algebras.
	\item If $B$ is of type I, then the C*-algebra $D \rtimes \G \ltimes B$ is also of type I.
	\end{enumerate}
	\end{prop}
	\begin{proof}
	(1) We only need to show that any representation of $D\rtimes \G \ltimes B$ decomposes into a direct sum of finite dimensional representations. To this end, we take a representation $\pi$ of $D \rtimes \G \ltimes B$ on a Hilbert space $H$ and take the associated covariant representation $(\pi_D,\pi_B,U)$. Since $\G$ is compact, $U$ decomposes into a direct sum of finite dimensional  irreducible representations: $H = \bigoplus_i H_i$. Then for each $\xi \in H_i$, its orbit $(D \rtimes \G \ltimes B) \xi = \pi_D(D) \pi_B(B) H_i$ is finite dimensional. By a simple maximality argument, we get the conclusion.

	(2) We fix a faithful $\G$-invariant state $\varphi_D$ on $D$. Since $D$ is finite dimensional, there exists $\lambda > 0$ such that for all $d \in D$ 
	\[\varphi_D(d^* d) 1 \geq \lambda d^* d.\]
	By Lemma \ref{lem ce}, the map
	\[E \colon D \rtimes \G \ltimes B \to \G \ltimes B \colon d x a \mapsto \varphi_D(d) x a\]
	defines a conditional expectation.
	Then $E(x^* x) \geq \lambda x^* x$ for any $x \in D \rtimes \G \ltimes B$. Therefore the conditional expectation $E^{**}$ from $(D \rtimes \G \ltimes B)^{**}$ to $(\G \ltimes B)^{**}$ is of finite index, hence $D \rtimes \G \ltimes B$ is of type I, as so is $\G \ltimes B \subset B \otimes K(L^2(G))$. The last argument uses the fact that in the separable context the type I property passes to any C*-subalgebra, as explained in the proof of \cite[Corollary 9.4.5]{BrownOzawa}.
	\end{proof}
	Similarly for a $\G$-equivariant Hilbert $B$-module $\cE$,	one can define a Hilbert $A \rtimes \G \ltimes B$-module $A \rtimes \G \ltimes \cE$ as a corner of the linking algebra $A \rtimes \G \ltimes \cK(\cE \oplus B)$.
	More concretely, the Hilbert module $A \rtimes \G \ltimes \cE$ is the completion of the pre-Hilbert module $\tilde \cE_0$ defined as follows:
	\begin{itemize}
	\item As a vector space, $\tilde \cE_0$ is isomorphic to $A_0 \atimes c_c(\dG) \atimes \cE$, where again $A_0$ denotes the respective algebraic core. Again the element in $\tilde \cE_0$ corresponding to $a \otimes x \otimes b$ is denoted by $a x b$ for $a \in A_0, x \in c_c(\dG), b \in \cE$.
	\item The right $A \rtimes \G \ltimes B$-module structure is given by
	\[(axb) (a' x' b') = a (x_{(1)} \triangleright a') x_{(2)} x'_{(1)} (b \triangleleft x'_{(2)}) b'\]
	for $a,a' \in A_0, x,x' \in c_c(\dG)$, $b \in \cE$ and $b' \in B_0$.
	\item The inner product is given by
	\[(bxa,b'x'a') = a^* x^* (b,b') x' a'\]
	for $a,a' \in A_0, x,x' \in c_c(\dG)$, $b,b' \in \cE$. Here $bxa$ expresses an element of $\tilde \cE_0$ by the same commutation relation as in $\mathcal{A}$.
	\end{itemize}
	It is easy to see that $\cK(A \rtimes \G \ltimes \cE)$ is naturally isomorphic to $A \rtimes \G \ltimes \cK(\cE)$. In particular we get the following result.
	\begin{lem}
	For each $\G$-equivariant Hilbert $B$-module $\cE$, $A \rtimes \G \ltimes \cK(\cE)$ is Morita equivalent to $A \rtimes \G \ltimes B$.
	\end{lem}

Finally we show that the construction preserves the exact sequences in a natural sense.

	\begin{lem}
	Let $A$ be a C*-algebra with a right $\G$-action $\alpha$. For a C*-algebra $B$ with a left $\G$-action $\beta$ and a $\G$-invariant ideal $I \subset B$, the sequence
	\[0 \to A \rtimes \G \ltimes I \to A \rtimes \G \ltimes B \to A \rtimes \G \ltimes (B/I) \to 0\]
	is exact.
	\end{lem}
	\begin{proof}
	Clearly the map $A \rtimes \G \ltimes B \to A \rtimes \G \ltimes (B/I)$ is surjective. To see the injectivity of $A \rtimes \G \ltimes I \to A \rtimes \G \ltimes B$, take a faithful nondegenerate representation $\pi$ of $A \rtimes \G \ltimes I$ on a Hilbert space $H$. Consider  the associated covariant representation $(\pi_A,\pi_I,U)$ and the unique extension of $\pi_I$ to $B$, denoted by $\pi_B$. Since
	\[U^*(1 \otimes \pi_I(b)) U = (\id \otimes \pi_I)\beta(b)\]
	for any $b \in M(I)$, the triple $(\pi_A,\pi_B,U)$ is a covariant representation. Since the associated representation $A \rtimes \G \ltimes B$ is an extension of $\pi$, we conclude that the map $A \rtimes \G \ltimes I \to A \rtimes \G \ltimes B$ is injective.

	It remains to prove that the sequence in the lemma is exact at the middle term. Since the composition is zero, we can induce a homomorphism
	\[(A \rtimes \G \ltimes B)/(A \rtimes \G \ltimes I) \to A \rtimes \G \ltimes (B/I).\]
	On the other hand, the universality of the crossed product induces the inverse of the map. Hence the homomorphism is an isomorphism.
	\end{proof}
	\begin{prop}\label{prop:functor}
	Let $\G$ be a compact quantum group with a countable dual.
	Fix a separable C*-algebra $A$ with a right $\G$-action. The crossed product introduced above gives rise to a triangulated functor on the equivariant Kasparov category
	\[A \rtimes \G \ltimes \cdot \colon KK^{\G} \to KK.\]
	\end{prop}
	\begin{proof}
	This is a direct consequence of \cite[Theorem 4.4]{NestVoigt} and the last two lemmas.
	\end{proof}

	\section{Adjunction}
	In this section we develop a construction which will allow us to establish a (natural) isomorphism of certain  equivariant and non-equivariant $KK$-groups.

	Let $D$ be a finite dimensional C*-algebra with a left $\G$-action, where $\G$ is again a compact quantum group (with $\GGamma$ its discrete dual). The opposite algebra $D^\op$ admits a right $\G$-action:
	\[\alpha^\op \colon D^\op \to D^\op \otimes C(\G) \colon x^\op \mapsto (\id \otimes R) \circ \sigma \circ \alpha(x)^\op.\]
	Our goal is to show the functor $D^\op \rtimes \G \ltimes \cdot \colon KK^{\G} \to KK$, developed in the previous section, is the right adjoint functor of $D \otimes \cdot \colon KK \to KK^{\G}$. This is done by constructing the counit and unit.

	From now on, we fix a faithful $\G$-invariant state $\varphi_D = {\rm Tr}(\rho {\cdot})$ on $D$, where ${\rm Tr}$ is the trace  taking value 1 at each minimal projection and $\rho \in D$ is the `density matrix' of $\varphi_D$.

	\subsection{Unit}
	The modular group of $\varphi_D$ is given by the formula $\sigma^{\varphi_D}_t(a) = \rho^{it} a \rho^{-it}$, $a \in D, t \in \mathbb{R}$. By \cite[Th\'{e}or\`{e}me 2.9]{Enock}, we get
	\[(\rho^{it} a \rho^{-it}) \triangleleft x = \rho^{it} (a \triangleleft \dtau_t(x)) \rho^{-it}, \;\; a \in D, t \in \mathbb{R},\]
	where $\dtau$ is the scaling group on $c_0(\dG)$. Let $p^0$ be the support of the counit on $c_0(\dG)$.
	\begin{lem}
	Let $D$ be a finite dimensional C*-algebra with a left $\G$-action.
	There exists $X \in D^\op \otimes D \subset M(D^\op \rtimes \G \ltimes D)$ such that
	\begin{enumerate}
	\item $X$ is positive;
	\item $a X = (\rho^{-1/2} a \rho^{1/2})^\op X$ and $X b = X(\rho^{1/2} b \rho^{-1/2})^\op$ for all $a,b \in D$;
	\item $p^0 X = X p^0$;
	\item $(\varphi_D^\op \otimes \id)(X) = 1$.
	\end{enumerate}
	\end{lem}
	\begin{proof}
	Since $D$ is finite dimensional, we write $D$ as a direct sum of matrix algebras
	\[D = \bigoplus_{\pi} M_{n(\pi)}\]
	and fix a matrix unit $(e^\pi_{ij})$ for each matrix algebra.
	Let
	\[X = \sum_{\pi,i,j} (\rho^{-1/2}  e^\pi_{ij} \rho^{-1/2})^\op e^\pi_{ji} \in D^\op \otimes D \subset M(D^\op \rtimes \G \ltimes D).\]

	The assertions (2) and (4) follow from a straightforward computation.

	For (1), recall that $\sum_{\pi,i,j} e^\pi_{ij} \otimes e^\pi_{ij} \in D \otimes D$ is positive. Now using the component-wise transpose map viewed as the an isomorphism ${\rm t} \colon D \to D^\op \colon e^\pi_{ij} \mapsto (e^\pi_{ji})^{\rm op}$, we conclude
	\[X = ((\rho^{-1/2})^\op \otimes 1) ({\rm t} \otimes \id)\left(\sum_{\pi,i,j} e^\pi_{ij} \otimes e^\pi_{ij} \right) ((\rho^{-1/2})^\op \otimes 1)\]
	is still positive.

	For (3), we define a vector space isomorphism $\iota \colon D \otimes D \to \End(D)$ by
	\[\iota(a \otimes b)(d) = \varphi_D(db) a.\]
	This is equivariant with respect to the tensor representation of $\G$ on $D \otimes D$ and the adjoint representation of $\G$ on $\End(D)$. Hence $\iota^{-1}(1_{\End(D)}) = \sum_{\pi,i,j} \rho^{-1} e^\pi_{ij} \otimes e^\pi_{ji}$ is invariant under the tensor representation.

	Now we observe that for all $x \in c_c(\GGamma)$
	\begin{align*}
	p^0 X x & = p^0 \sum_{\pi,i,j} \left((\rho^{-1/2} e^\pi_{ij} \rho^{-1/2}) \triangleleft \dtau_{-i/2}(x_{(2)})\right)^\op \left(e^\pi_{ij} \triangleleft x_{(1)}\right)\\
	& = p^0 \sum_{\pi,i,j} \left(\rho^{1/2}(\rho^{-1}e^\pi_{ij}) \triangleleft x_{(2)} \rho^{-1/2} \right)^\op \left(e^\pi_{ij} \triangleleft x_{(1)} \right) \\
	& = p^0 \de(x) X.
	\end{align*}
	In particular $p^0 X p^0 = p^0 X$. Since $X$ is self-adjoint, we get $ X p^0 = p^0 X p^0 = p^0 X$.
	\end{proof}
	Thanks to the lemma above, the vector space $D$ admits a right Hilbert $D^\op \rtimes \G \ltimes D$-module structure defined as follows:
	\begin{itemize}
	\item 	The right module structure is given by $d \triangleleft (a^\op x b) = ((\rho^{-1/2}a \rho^{1/2}d) \triangleleft x) b$ for $a,b,d\in D$, $x \in c_c(\dG)$.
	\item The $D^\op \rtimes \G \ltimes D$-valued inner product is given by
	\[(d,d') = d^* X p^0 d',\;\;\; d, d' \in D.\]
	\end{itemize}
	We denote $D$ equipped with the right Hilbert $D^\op \rtimes \G \ltimes D$-module structure above by $\cD$.
	\subsection{Counit}
	Suppose now that we also have a C*-algebra $B$ equipped with a right action $\beta$ of $\G$, with the algebraic core $B_0$.
	Recall that we have a conditional expectation $D^\op \rtimes \G \ltimes B \to \G \ltimes B$ as in Lemma \ref{lem ce}.
	Composing $D^\op \rtimes \G \ltimes B \to \G \ltimes B$ with the natural operator-valued weight from $ \G \ltimes B$ to $B$, we get an operator-valued weight $E \colon D^\op \rtimes \G \ltimes B \to B$. On $D^\op \rtimes_{\rm alg} \G \ltimes_{\rm alg} B_0$ it is given by
	\[E \colon D^\op \rtimes \G \ltimes B \to B \colon a x b \mapsto \varphi_D^\op(a) \dphi(x) b,\]
	with $\dphi$ denoting  the right Haar weight of $\GGamma$.
	The corresponding GNS module is isomorphic to $\cE_B := L^2(D) \otimes L^2(\G) \otimes B$ with the GNS map
	\[\Lambda \colon D^\op \rtimes_{\rm alg} \G \ltimes_{\rm alg} B \to L^2(D) \otimes L^2(\G) \otimes B \colon d^\op b x \mapsto d \Omega \otimes \Lambda_{\dphi}(x)  \otimes b\]
	for $d \in D, x \in c_c(\dG), b \in B$.
	Furthermore $\cE_B$ carries a natural $\G$-equivariant $D \otimes (D^\op \rtimes \G \ltimes B)$-$B$-bimodule structure defined as follows:

	\begin{itemize}
	\item The $\G$-action $\beta_{\cE_B}$ on $\cE_B$ is given by
	\[\beta_{\cE_B}(x) = U_{12}^* V_{13}^* (\id \otimes \id \otimes \beta)(x), \;\;\; x \in \cE_B.\]
	\item The left action of $D \otimes (D^\op \rtimes \G \ltimes B)$ structure is given by
	\begin{enumerate}
	\item $D \otimes 1 \ni d \otimes 1 \mapsto \lambda_D(d) \otimes 1 \otimes 1$
	\item $1 \otimes D^\op \ni 1 \otimes d^\op \mapsto (\rho_D \otimes \rho)\alpha^\op(d^\op) \otimes 1$,
	\item $1 \otimes c_c(\dG) \ni 1 \otimes x \mapsto 1 \otimes \dlambda(x) \otimes 1$,
	\item $1 \otimes B_0 \ni 1 \otimes b \mapsto 1 \otimes (\lambda \otimes \id)\beta(b)$.
	\end{enumerate}
	\end{itemize}
	From this presentation, it follows that the image of $D \otimes (D^\op \rtimes \G \ltimes B)$ is in $K(\cE_B)$.	
	\subsection{Adjunction}
	We begin by stating a general lemma.
	
	\begin{lem}
	Let $A,B,C$ be C*-algebras, $\cE$ a right Hilbert $A$-module, $\cF$ a Hilbert $A \otimes B$-$C$-bimodule.
	Then we have a natural isomorphism
	\[(\cE \otimes B) \otimes_{A \otimes B} \cF \simeq \cE \otimes_A \cF.\]
	\end{lem}
\begin{proof}
Direct computation.
\end{proof}

	\begin{lem}\label{adj1}
	There exists a unitary
	\[U \colon \cD \otimes_{D^\op \rtimes \G \ltimes D} \cE_D \to D\]
	defined by
	\[U(d \otimes_{D^\op \rtimes \G \ltimes D}  \Lambda(x)) = \rho (d \triangleleft x)\]
	for $d \in D$, $x \in D^\op \rtimes \G \ltimes D$.
	\end{lem}
	\begin{proof}
	We only need to show that $U$ is an isometry.
	This is done by a straightforward computation. Indeed,
	take $x,x' \in c_c(\dG)$, $a^\op, a'^\op \in D^\op$ and $b, b' \in D$. Then
	\begin{align*}
	(d \otimes_{D^\op \rtimes \G \ltimes D}  \Lambda(xa^\op b),d' \otimes_{D^\op \rtimes \G \ltimes D}  \Lambda(x' a'^\op b')) & = E(b^* (a^\op)^* x^* d^* X p^0 d' x' a'^\op b') \\
	& = E(b^* (a^\op)^* (d \triangleleft x)^* X p^0 (d' \triangleleft x') a'^\op b) \\
	& = E(b^* (d \triangleleft x)^* \rho^{1/2} a^* \rho^{-1/2} X p^0 \rho^{-1/2} a' \rho^{1/2} (d' \triangleleft x') b) \\
	& = (d \triangleleft x a^\op b) \rho^2 (d' \triangleleft x' a'^\op b').
	\end{align*}
	\end{proof}
	\begin{lem}\label{adj2}
	There exists a unitary
	\[V \colon \cD \otimes_{D^\op \rtimes \G \ltimes D} D^\op \rtimes \G \ltimes \cE_D \simeq D^\op \rtimes \G \ltimes D\]
	defined by
	\[V(d \otimes_{D^\op \rtimes \G \ltimes D} a^\op x \Lambda(b^\op y c)) = (d \triangleleft a^\op x) b^\op y c\]
	\end{lem}
	\begin{proof}
	Since
	\[(d \otimes (a^\op x \Lambda(b^\op y c)) = 1 \otimes \Lambda (y) (\rho^{1/2} (d \triangleleft a^\op x \rho^{1/2} b \rho^{-1/2}) \rho^{-1/2})^\op c\]
	for $a,b,c, d \in D$ and $x,y \in c_c(\dG)$,
	we only need to show
	\[(1 \otimes \Lambda(x), 1 \otimes \Lambda(y)) = x^*y.\]
	To this end, we compute, using the antipode of $c_c(\GGamma)$ denoted $\hat{S}$,
	\begin{align*}
	(1 \otimes \Lambda(x), 1 \otimes \Lambda(y)) & = (\Lambda(x),X p^0 \Lambda(y)) \\
	& = \sum_{\pi,i,j} (\Lambda(x), \Lambda(\rho^{1/2} e^\pi_{ji} \rho^{-1/2})^\op  y \dS^{-1}(p^0_{(2)}) (\rho^{1/2} e^\pi_{ij} \rho^{-1/2}) p^0_{(1)}) \\
	& = \sum_{\pi,i,j} \varphi_D(\rho^{1/2} e^\pi_{ij} \rho^{-1/2}) \dphi(x^*y \dS^{-1}(p^0_{(2)})) p^0_{(1)} \\
	& = x^* y.
	\end{align*}
	Here we have used the identity
	\[p^0_{(1)} x \otimes p^0_{(2)} = p^0_{(1)} x_{(1)} \otimes p^0_{(2)} x_{(2)} \dS(x_{(3)}) = p^0_{(1)} \otimes p^0_{(2)} \dS(x).\]
	\end{proof}
	We have two functors
	\[\KK \to \KK^{\G} \colon A \mapsto D \otimes A, \KK^{\G} \to \KK \colon A \mapsto D^\op \rtimes \G \ltimes A.\]
	\begin{thm}\label{thm:adjoint}
	Let $\G$ be a compact quantum group with a countable dual. For $A \in \KK$ and $B \in \KK^{\G}$,
	we have natural isomorphisms
	\[KK^{\G}(D \otimes A, B) \simeq KK(A, D^\op \rtimes \G \ltimes B).\]
	\end{thm}
	\begin{proof}
	We construct the counit-unit adjunction.
	The unit is given by
	\[\eta_A = [\cD] \otimes 1_A \in KK(A, (D^\op \rtimes \G \ltimes D) \otimes A)\]
	and the counit is given by 
	\[\varepsilon_B = [\cE_B] \in KK^{\G}(D \otimes (D^\op \rtimes \G \ltimes B), B).\]

	We need to prove
	\[\varepsilon_{D \otimes A} (\id_D \otimes \eta_{A}) = \id_{D \otimes A}, (D^\op \rtimes \G \ltimes \varepsilon_B) \eta_{D^\op \rtimes \G \ltimes B} = \id_{D^\op \rtimes \G \ltimes B}.\]
	The first identity is due to Lemma \ref{adj1}:
	\[(D \otimes \cD \otimes A) \otimes_{D \otimes (D^\op \rtimes \G \ltimes D) \otimes A} (\cE_D \otimes A) \simeq D \otimes A.\]
	This follows from $\cD \otimes_{D^\op \rtimes \G \ltimes D} \cE_D \simeq D$.
	The second identity is due to Lemma \ref{adj2}:
	\[(\cD \otimes D^\op \rtimes \G \ltimes B) \otimes_{(D^\op \rtimes \G \ltimes D) \otimes (D^\op \rtimes \G \ltimes B)} (D^\op \rtimes \G \ltimes \cE_B) \simeq D^\op \rtimes \G \ltimes B.\]
	\end{proof}

	\section{Application to UCT}\label{section UCT}
In this section we will apply the last theorem to questions regarding the Universal Coefficient Theorem. Let $\G$ be again a compact quantum group.

A $\G$-C*-algebra is said to be \emph{cofibrant} if it is of the form $D \otimes A$ where $D$ is a finite dimensional C*-algebra with a $\G$-action $\alpha$ and the $\G$-action on $D \otimes A$ is given by $\alpha \otimes \id$. Let ${\rm Cof}$ be the full subcategory of cofibrant objects in $KK^{\G}$ and let $\mathcal{N}$ be the full subcategory of $A \in KK^{\G}$ such that $D^\op \rtimes \G \ltimes A$ is $KK$-contractible for any finite dimensional $\G$-C*-algebra $D$.
	\begin{cor} Suppose that the dual of $\G$ is countable.
	The subcategories $(\langle {\rm Cof} \rangle, \mathcal{N})$ are complementary, i.e., for any $A \in KK^{\G}$, there exists a unique triangle
	\[\xymatrix{& P(A) \ar[ld] & \\ A \ar[rr] && N(A)\;\;\;, \ar[lu]|\circ }\]
	where $P(A) \in \langle {\rm Cof} \rangle$ and $N(A) \in \mathcal{N}$.
	\end{cor}
	\begin{proof}
	By the help of Theorem \ref{thm:MN} and Theorem \ref{thm:adjoint}, we only need to show the isomorphism class of finite dimensional $\G$-C*-algebras is at most countable. First from \cite{YukiKenny}, the isomorphism classes of finite dimensional $\G$-C*-algebra are in one-to-one correspondence with the $Q$-systems in ${\rm Rep}(\G)$. There exists only countably many objects in ${\rm Rep}(\G)$ and each of them has at most finitely many structures of a $Q$-system by \cite{IzumiKosaki}.
	\end{proof}


 Recall now the phantom tower construction.

	For a separable $\G$-C*-algebra $A$, define $P_n, N_n \in KK^{\G}$ inductively as follows:
	\begin{itemize}
	\item Put $A = N_0$.
	\item For $N_n$, we set $P_{n+1} = \bigoplus_{D :\text{torsion}} D \otimes (D^\op \rtimes \G \ltimes N_n)$. Then we have the counit morphism $\bigoplus_D \varepsilon_{N_n} \colon P_{n+1} \to N_n$. We embed this morphism into a triangle $P_{n+1} \to N_n \to N_{n+1} \to SP_{n+1}$ to define $N_{n+1}$.
	\end{itemize}
	With the construction above, we get a diagram:
	\[\xymatrix{& P_1 \ar[ld] & & P_2 \ar[ll]|\circ \ar[ld] && \ar[ll]|\circ \ar[ld] \\ A \ar[rr] && N_1 \ar[lu]|\circ \ar[rr] && N_2 \ar[lu]|\circ \ar[rr] &&& }\]
	Now consider the morphism $A \to N_n$ to fit in a triangle $A \to N_n \to \tilde A_n$.
	The octahedral axiom shows that $\tilde A_n$ also fits to a triangle
	\begin{align}\label{eqn:tri}
	\tilde A_n \to \tilde A_{n+1} \to P_n \to S \tilde A_n.
	\end{align}
	We take the homotopy limit $N = {\rm ho \mathchar`- lim} N_n$ and $\tilde A = {\rm ho \mathchar`- lim} \tilde A_n$.
	Then we have a triangle
	\[\xymatrix{& \tilde A \ar[ld] & \\ A \ar[rr] && N \;\;\;\;,\ar[lu]|\circ }\]
	where $\tilde A \in \langle {\rm Cof} \rangle$ and $N \in \mathcal{N}$.
	
	\begin{thm}\label{thm:UCT}
		Let $\G$ be a compact quantum group with a countable dual and let $A$ be a separable $\G$-C*-algebra. Suppose that either 
\begin{rlist}
\item $D^\op \rtimes \G \ltimes A$ satisfies the UCT for any finite dimensional $\G$-C*-algebra $D$, or
\item $\G$ has no torsion and 	 $\G \ltimes A$ satisfies the UCT. 
\end{rlist}		
 Then $P(A)$ satisfies the UCT.  
	\end{thm}
	\begin{proof}
Note first that the condition  (ii) implies (i) as  if $\G$ has no torsion, then its action on $D$ is Morita equivalent to the trivial action.
Thus we only need to prove (i).
		
	First, by induction, we show that $D^\op \rtimes \G \ltimes N_n$ satisfies the UCT for any $n \in \mathbb{N}\cup \{0\}$. This holds for $n=0$ by assumption. Assume $D^\op \rtimes \G \ltimes N_n$ satisfies the UCT. Since $D^\op \rtimes \G \ltimes D$ is a direct sum of matrix algebras, $D^\op \rtimes \G \ltimes P_{n+1} \simeq (D^\op \rtimes \G \ltimes D) \otimes (D^\op \rtimes \G \ltimes N_n)$ satisfies the UCT. Hence the mapping cone $D^\op \rtimes \G \ltimes N_{n+1}$ also satisfies the UCT.

	In particular, $P_n$ satisfies the UCT.
	Again by induction we see that $\tilde A_n$ satisfies the UCT for any $n\in \mathbb{N}$ by $\tilde A_1 = P_1$ and the triangle (\ref{eqn:tri}). Passing to the homotopy limit, we get that $P(A)$ satisfies the UCT.
	
The last two paragraphs could be replaced by observing that the operations generating the localising subcategory preserve the propert of "being in the UCT class".
	\end{proof}
	\begin{df} Fix a discrete quantum group $\QG$ with the dual compact quantum group $\G$, and consider $\langle {\rm Cof} \rangle$ and $ \mathcal{N}$, the subcategories of $KK^\G$  introduced above.
	We say that $\QG$ satisfies the \emph{$\langle {\rm Cof} \rangle$-Baum--Connes property} if $N$ is $KK$-contractible for any $N \in \mathcal{N}$, or equivalently, $P(A) \to A$ is a $KK$-equivalence.

	We say that $\QG$ satisfies the \emph{strong $\langle {\rm Cof} \rangle$ Baum--Connes property} if $N$ is $KK^{\G}$-contractible for any $N \in \cN$, or equivalently,  $KK^{\G} = \langle {\rm Cof} \rangle$.
	\end{df}
	\begin{rem}
	When $\QG$ is a classical discrete group, then the $\langle {\rm Cof} \rangle$-Baum--Connes property is equivalent to the fact that the strong Baum--Connes conjecture as introduced in \cite[Definition 9.1]{MeyerNest} holds, by \cite[Theorem 9.3]{MeyerNest}. We do not know whether $KK^{\G} = \langle {\rm Cof} \rangle$ even when $\QG$ is a  finite group.

In spite of that, many discrete quantum groups actually satisfy the strong $\langle {\rm Cof} \rangle$-Baum--Connes property. In particular, it holds for duals of compact connected groups \cite{MeyerNest2}, for  free orthogonal quantum groups \cite{VoigtBC}, free unitary quantum groups \cite{VergniouxVoigt} and free permutation groups \cite{VoigtQAut}. It passes through the monoidal equivalence and is closed under taking subgroups and free wreath product \cite{FreslonMartos}.
It also seems plausible that one could exploit results of  \cite{VergniouxVoigt} to deduce that it is also stable under free products.
	\end{rem}

	\begin{cor} \label{Cor:UCT} Let $A$ be a separable $\dG$-C*-algebra. We have the following.
	\begin{enumerate}
	\item For any torsion-free countable discrete quantum group $\dG$ with the $\langle {\rm Cof} \rangle$-Baum--Connes property, the C*-algebra $\dG \ltimes A$ satisfies the UCT if $A$ does.
	\item For any countable discrete quantum group $\dG$ with the $\langle {\rm Cof} \rangle$-Baum--Connes property, the C*-algebra $\dG \ltimes A$ satisfies the UCT if $A$ is of type I. In particular the reduced group C*-algebra $C(\G)$ satisfies the UCT.
	\end{enumerate}
	\end{cor}
	\begin{proof} Recall that we denote the dual of $\QG$ by $\G$.

	(1) Consider the $\G$-C*-algebra $\QG \ltimes A$. Then by the Baaj--Skandalis duality (see \cite{Vaes} and recall that discrete/compact quantum groups are automatically regular), $\G \ltimes \QG \ltimes A \simeq K(L^2(\G)) \otimes A$, so that $\G \ltimes \QG \ltimes A$ satisfies the UCT. Now apply Theorem \ref{thm:UCT} to show $P(\QG \ltimes A)$ satisfies the UCT. By $\langle {\rm Cof} \rangle$-Baum--Connes property, the C*-algebra $\QG \ltimes A$ also satisfies the UCT.

	(2) We denote the $\G$-C*-algebra $\QG \ltimes A$ by $B$. First we will show that $D^\op \rtimes \G \ltimes B$ is of type I.
	Again by the Baaj--Skandalis duality, the C*-algebra $\G \ltimes B \simeq K(L^2(\G)) \otimes A$ is of type I.
	Hence by Proposition \ref{crossed fd}, $D^\op \rtimes \G \ltimes B$ is of type I. In particular $D^\op \rtimes \G \ltimes B$ satisfies the UCT. The rest of the proof is the same as (1).
	\end{proof}

\section{Quantum Rosenberg Conjecture}

One of very well-known conjectures regarding group \cst-algebras, the Rosenberg Conjecture, stating that reduced \cst-algebras of a countable amenable discrete groups are quasidiagonal, was established in \cite{TWWUCT} (with the converse implication proved much earlier by Rosenberg in \cite{Ros}). Here we show how as a corollary of that result and the progress on the UCT conjecture for quantum group algebras made in the last section  one can obtain a similar statement for a large class of (unimodular) discrete quantum groups. A key  observation is that the proof of the original result of Rosenberg 
found  by Davidson (\cite{Dav} -- we refer to this book also for the definition of quasidiagonality) passes to the quantum case in a straightforward manner. 

\begin{thm}
	Assume that $\GGamma$ is a countable discrete quantum group. Then quasidiagonality of $\cst_r(\GGamma)$ implies amenability of $\GGamma$ and if $\GGamma$ is unimodular, amenable and satisfies the $\langle \textup{Cof}\rangle$-Baum-Connes property then $\cst_r(\GGamma)$ is quasidiagonal. Finally there exist amenable nonunimodular countable discrete quantum groups $\GGamma$ (e.g. $\widehat{SU_q(2)}$, $q \in (0,1)$) such that  $\cst_r(\GGamma)$ is not quasidiagonal.
\end{thm}

\begin{proof}
	Assume first that $\cst_r(\GGamma)$ is a quasidiagonal \cst-algebra.
	
	If $\GGamma$ is finite, then it is obviously amenable (in that case it is compact and the Haar state yields an invariant mean). If $\GGamma$ is infinite, then the left regular representation of $\cst_r(\GGamma)$ is essential, i.e.\ contains no compact operators (\cite{Kal}). This means (via one of the versions of the Voiculescu Theorem) that $\cst_r(\GGamma)$ is quasidiagonal as the set of operators in $B(\ltwo)$. Let then $(P_n)_{n=1}^{\infty}$ be a sequence of finite rank projections in $B(\ltwo)$ increasing to $I$ and such that for each $\alpha \in  \Irr_{\hQG}$, $i,j=1,\ldots,n_{\alpha}$, we have
	\[ \|P_n u^{\alpha}_{i,j} -  u^{\alpha}_{i,j} P_n\|\stackrel{n\to \infty}{\longrightarrow} 0.\]
	Let then $\tau_n$ be the normalised trace on the matrix algebra $B(P_n \ell^2(\QG))$, fix an ultrafilter $\mathcal{U}$ on $\bn$ and define the state $\omega$ on $B(\ell^2(\QG))$  via the prescription
	\[ \omega(a) = \lim_{\mathcal{U}} \tau_n (P_n a P_n), \;\;\; a \in  B(\ell^2(\QG)),\]
	and let $m=\omega|_{l^{\infty} (\QG)}$. Fix $\alpha \in \Irr_{\hQG}$, $i,j\in \{1,\ldots,n_\alpha\}$ and a non-zero $x \in l^{\infty}(\QG)$. Put $\phi = \widehat{e_{i,j}^{\alpha}} \in  {l^{\infty} (\QG)}^*$. Note that $\phi(1) = \delta_{i,j}$.
	
	For each $\epsilon>0$ we can find $N\in \bn$ such that for all $n\geq N$ and $ p=1,\ldots,n_{\alpha}$ we have
	\[ \|P_n u^{\alpha}_{i,p} - u^{\alpha}_{i,p} P_n\| \leq \epsilon (3 n_{\alpha} \|x\|)^{-1}\]
	and
	\[ \|P_n u^{\alpha}_{j,p} - u^{\alpha}_{j,p} P_n\| \leq \epsilon (3 n_{\alpha} \|x\|)^{-1}\]
	(note that the latter estimate is valid if one replaces $u^{\alpha}_{j,p}$ by $(u^{\alpha}_{j,p})^*$). 
	Thus for such $n\geq N$  (see \eqref{Lphi})
	\begin{align*} \left|\tau_n (P_n L_{\phi} (x) P_n) - \right.& \left. \delta_{i,j} \tau_n (  P_n x P_n) \right| = \left| \tau_n \left(P_n \left(\sum_{p=1}^{n_{\alpha}} u^{\alpha}_{i,p}  x (u^{\alpha}_{j,p})^*\right) P_n \right) - \delta_{i,j} \tau_n (  P_n x P_n) \right|
	\\&\leq \left| \sum_{p=1}^{n_{\alpha}} \tau_n (  P_n  u^{\alpha}_{i,p} P_n x P_n (u^{\alpha}_{j,p})^* P_n )- \delta_{i,j} \tau_n (  P_n x P_n) \right|  + \frac{2 \epsilon}{3}
	\\&= \left|
	\sum_{p=1}^{n_{\alpha}} \tau_n (  P_n x P_n (u^{\alpha}_{j,p})^* P_n   u^{\alpha}_{i,p}P_n)- \delta_{i,j} \tau_n (  P_n x P_n) \right| + \frac{2 \epsilon}{3}
	\\& \leq \left|\sum_{p=1}^{n_{\alpha}} \tau_n (  P_n x P_n (u^{\alpha}_{j,p})^* u^{\alpha}_{i,p}P_n) - \delta_{i,j} \tau_n (  P_n x P_n) \right|+ \epsilon =  \epsilon,
	\end{align*}
	where in the third equality we used the fact that $\tau_n$ is a trace and in the last one the unitarity of the matrix $(u^{\alpha}_{i,j})_{i,j=1}^{n_{\alpha}}$.
	This implies that in the limit we obtain
	\[ m(L_{\phi}(x)) = \lim_{\mathcal{U}} \tau_n (P_n L_{\phi}(x) P_n) = \delta_{i,j} m(x) = \phi(1) m(x)\]
	and the proof of the forward implication is finished.
	
	Assume then that $\GGamma$ is amenable and unimodular. By Theorem 1.1 of \cite{BMT}  $\cst_r(\GGamma)$ is nuclear; by the  unimodularity of $\GGamma$, $\cst_r(\GGamma)$ admits a faithful trace. Then the main result of \cite{TWWUCT} shows that $\cst_r(\GGamma)$ is quasidiagonal if it satisfies the UCT. This however follows from the assumption that $\GGamma$ satisfies the $\langle \textup{Cof}\rangle$-Baum-Connes property by Corollary \ref{Cor:UCT} (2).
	
	It remains to note that Woronowicz's compact quantum group $SU_q(2)$ (with $q \in (0,1)$) is coamenable, as first noted in \cite{Banica}. Thus  $\widehat{SU_q(2)}$ is amenable. On the other hand the \cst-algebra $\cst_r(\widehat{SU_q(2)})$, which is in fact independent of $q$, as observed in Theorem A.2 in \cite{Wor2}, contains a proper isometry, so in particular cannot be quasidiagonal.
\end{proof}

Note that if we knew that all group \cst-algebras of discrete amenable unimodular quantum groups satisfy UCT we could drop the $\langle \textup{Cof}\rangle$-Baum-Connes property assumption in the second part of the theorem above.


\end{document}